%% file: ANTS2020_0308_arXiv.tex
\documentclass{amsart}
\usepackage{graphics}
\input{preamble.tex}

\newtheorem{theorem}{Theorem}[section]
\newtheorem{lemma}[theorem]{Lemma}
\newtheorem{proposition}{Proposition}[section]

\theoremstyle{definition}

\newtheorem{example}[theorem]{Example}

\newtheorem{algorithm}[theorem]{Algorithm}
\newtheorem{problem}[theorem]{Problem}

\theoremstyle{remark}
\newtheorem{remark}[theorem]{Remark}

\numberwithin{equation}{section}

\begin{document}

\title{Algorithm to enumerate superspecial Howe curves of genus $4$}


\author{Momonari Kudo}
\address{Kobe City College of Technology, 8-3, Gakuen-Higashimachi,
Nishi-ku, Kobe 651-2194, Japan}
\curraddr{}
\email{m-kudo@math.kyushu-u.ac.jp}
\thanks{}

\author{Shushi Harashita}
\address{Graduate School of Environment and Information Sciences,
Yokohama National University, 79-7, Tokiwadai, Hodogaya-ku,
Yokohama 240-8501, Japan}
\curraddr{}
\email{harasita@ynu.ac.jp}
\thanks{}



\date{}

\dedicatory{}

\begin{abstract}
A Howe curve is a curve of genus $4$ obtained as the fiber product over $\bbP^1$ of two elliptic curves. Any Howe curve is canonical.
This paper provides an efficient algorithm to find
superspecial Howe curves and that to enumerate
their isomorphism classes.
We discuss not only an algorithm to test the superspeciality but also an algorithm to test isomorphisms for Howe curves.
Our algorithms are much more efficient than conventional ones proposed by the authors so far for general canonical curves.
We show the existence of a superspecial Howe curve in characteristic $7<p\le 331$
and enumerate the isomorphism classes of superspecial Howe curves in characteristic $p\le 53$, by executing our algorithms over the computer algebra system Magma.
\end{abstract}

\maketitle

\input{1.tex}
\input{2and3.tex}
\input{4.tex}

\input{5and6_v2.tex}

\subsection*{Acknowledgments}
This work was supported by JSPS Grant-in-Aid for Scientific
Research (C) 17K05196, and JSPS Grant-in-Aid for Research Activity Start-up 18H05836 and 19K21026.
The authors thank Masaya Yasuda for his comments which improved the exposition.
The authors also thank Kazuhiro Yokoyama for his helpful suggestions and comments on algorithm details and complexity analysis.

\end{document}

%% file: preamble.tex
\usepackage{amsthm}
\usepackage{amssymb}
\usepackage{amscd}
\usepackage{amsmath}
\usepackage{color} %


\def\F{\mathbb{F}}

\def\deg{{\rm deg}}

\def\dim{{\rm dim}}
\def\rank{{\rm rank}}

\newcommand{\bbP}{{\operatorname{\bf P}}}

\newcommand{\Proj}{\operatorname{Proj}}

\newcommand{\PGL}{\operatorname{PGL}}

\newcommand{\Aut}{\operatorname{Aut}}

\newcommand{\modulo}{\operatorname{mod}}

\newcommand{\Jac}{\operatorname{Jac}}

%% file: 1.tex
\section{Introduction}\label{sec:intro}
\subsection{Background and motivation}
Let $K$ be a field of characteristic $p>0$.
A nonsingular curve over $K$ is called {\it superspecial} (resp.\ {\it supersingular}) if its Jacobian variety is isomorphic (resp.\ isogenous) to a product of supersingular elliptic curves over an algebraically closed field containing $K$.
It is known that a nonsingular curve $X$ over $K$ is superspecial if and only if its Hasse-Witt matrix (resp.\ Cartier-Mainin matrix), which represents the Frobenius (resp.\ Cartier) operator on the cohomology group $H^1 (X, \mathcal{O}_X)$ (resp.\ $H^0(X,\Omega_X^1)$), is zero, where $\mathcal{O}_X$ (resp.\ $\Omega_X^1$) is the structure sheaf on $X$ (resp.\ the sheaf of differential $1$-forms on $X$).
Superspecial curves are not only theoretically interesing in algebraic geometry and number theory but also have many applications in coding theory, cryptology and so on, because they tend to have many rational points and their Jacobian varieties have big endomorphism rings.
However, it is not easy to find such curves since there are only finitely many superspecial curves for given genus and characteristic.
Our basic problem is to find or enumerate superspecial curves.
A method to construct superspecial curves is to consider fiber products of superspecial curves of lower genera.
In this paper, we illustrate that this method is efficient, by considering the simplest one among cases of genus $\ge 4$, that is the case of Howe curves. A {\it Howe curve} is a curve of genus $4$ obtained by fiber products over $\bbP^1$ of two elliptic curves $E_1$ and $E_2$.
In \cite{Howe}, E.\ Howe studied these curves to construct quickly curves with many rational points.

\subsection{Related works}
The reason that we consider the case of genus $g\ge 4$ is that the enumeration of the isomorphism classes of superspecial curves of $g\le 3$ has already been done, see Deuring~\cite{Deuring} for $g=1$, Ibukiyama-Katsura-Oort \cite{IKO} for $g=2$ and Brock~\cite{Brock} for $g=3$, also see Ibukiyama~\cite{Ibukiyama} and Oort~\cite{Oort91} for the existence for $g=3$.
Contrary to the case $g\le 3$, the existence/non-existence of a superspecial curve of genus $4$ in general characteristic is an open problem, while some results for small concrete $p$ are known;
See \cite[Theorem 1.1]{Ekedahl} for the non-existence for $p\le 3$ and \cite[Theorem B]{KH17} for the non-existence for $p=7$.
As for enumeration, computational approaches have been proposed recently in \cite{KH17}, \cite{KH18} and \cite{KH19}, and the main strategy common to \cite{KH17} and \cite{KH18} (resp.\ \cite{KH19}) is:
Parameterizing canonical (resp.\ hyperelliptic) curves $X$ of genus $4$, finding superspecial curves is reduced into computing zeros of a multivariate system derived from the condition that the Hasse-Witt (resp.\ Cartier-Manin) matrix of $X$ is zero.
With computer algebra techniques such as Gr\"{o}bner bases, the authors enumerated superspecial canonical (resp.\ hyperelliptic) curves for $p \leq 11$ (resp.\ $p \leq 23$) in \cite{KH17} and \cite{KH18} (resp.\ \cite{KH19}).
However, results for larger $p$ have not been obtained yet due to the cost of solving multivariate systems, and no complexity analysis is given in \cite{KH17}, \cite{KH18} and \cite{KH19}.

Now we restrict ourselves to Howe curves.
Recently, it is proven in \cite{KHS20} that there exists a supersingular Howe curve in arbitrary positive characteristic.
In particular, the authors of \cite{KHS20} reduced the existence of such a curve into that of a certain zero of a multivariate system as follows:
Given two elliptic curves $E_{A_i,B_i}:y^2 = x^3 + A_ix + B_i$ with $A_i, B_i \in \overline{\mathbb{F}_{p}}$ for $i = 1$ and $2$, they constructed a family of Howe curves $H$ (the normalization of $E_1\times_{\bbP^1} E_2$) with three parameters $\lambda$, $\mu$ and $\nu$ in $\overline{\mathbb{F}_{p}}$, where each $E_i$ is an elliptic curve with $E_i \cong E_{A_i,B_i}$ over $\overline{\mathbb{F}_{p}}$.
They also proved that the supersingularity of $H$ is equivalent to that of $E_1$, $E_2$ and a curve $C$ of genus $2$ if $p\ne 2$, by using the fact (due to Howe \cite[Theorem 2.1]{Howe}) that there exists an isogeny of degree $2^4$ from the Jacobian $J (H)$ to $E_1 \times E_2 \times J(C)$.
Thus, once supersingular $E_{A_i, B_i}$'s are given, the supersingularity of $H$ is reduced into that $(\lambda, \mu, \nu)$ is a certain zero over $\overline{\mathbb{F}_{p}}$ of a multivariate system derived from the supersingularity of $C$.
To show the existence of such a zero $(\lambda, \mu, \nu)$, they also find various algebraic properties of the defining polynomials of the system.
Based on these properties, they completed the proof by showing an analogous result of the quasi-affineness of Ekedahl-Oort strata in the case of abelian varieties.

The above reduction is applicable also for the superspecial case, but the method of \cite{KHS20} to prove the existence could not work well.
For this reason, the superspecial case is still open:
\begin{problem}\label{prob:sspHowe}
Does there exist a superspecial Howe curve for any $p>7$? 
\end{problem}
To obtain a general result for solving Problem \ref{prob:sspHowe}, it is meaningful and useful to obtain computational results for concrete $p$ as large as possible.

\subsection{Our contribution}
Here, we state our main results (Theorems \ref{thm:main1} and \ref{thm:main2} below).
The proofs are computational, and it has three main ingredients (A), (B) and (C) below, where we construct algorithms implemented over the computer algebra system Magma~\cite{Magma}.
In particular, the ingredient (C) provides both theoretically new and computationally efficient method to test for isomorphisms of Howe curves.
The first theorem is on Problem \ref{prob:sspHowe}, and the second one is on the enumeration of superspecial Howe curves.

\begin{theorem}\label{thm:main1}
There exists a superspecial Howe curve in characteristic $p$ if $7<p\le 331$.
\end{theorem}

\begin{theorem}\label{thm:main2}
We enumerate the isomorphism classes of superspecial Howe curves in characteristic $p\le 53$, see Table \ref{table:1}.
\end{theorem}

Each theorem holds in any characteristic bounded by a constant (which can be updated).
The constant is relatively large, compared with $p \le 7$ in \cite{KH17}, $p\le 11$ in \cite{KH18}, and $p\le 23$ in \cite{KH19}.
That shows an advantage of the use of Howe curves (with our algorithms below).

The three key ingredients are the following (A), (B) and (C):
\paragraph{\bf (A) Reduction of the search space.}
We discuss the field of definition of superspecial Howe curves (cf.\;Proposition \ref{prop:field_of_def}), which enables us to reduce the amount of our search drastically.
Specifically, the coordinates of $(A_1, B_1, A_2, B_2, \lambda, \mu, \nu)$ belong to $\mathbb{F}_{p^2}$, whereas those in the supersingular case~\cite{KHS20} can belong to larger extension fields in the algebraic closure $\overline{\mathbb{F}_{p}}$.

\paragraph{\bf (B) Algorithm to list superspecial Howe curves.}
Once the search space for $(A_1, B_1, A_2, B_2, \lambda, \mu, \nu)$ is reduced from $(\overline{\mathbb{F}_{p}})^7$ to $(\mathbb{F}_{p^2})^7$ by (A),
we develop an algorithm to list $(A_1, B_1, A_2, B_2, \lambda, \mu, \nu) \in (\mathbb{F}_{p^2})^7$ for which $H$ is superspecial, as well as our previous works~\cite{KH17}, \cite{KH18}, \cite{KH19}.
Different from \cite{KH17}, \cite{KH18}, \cite{KH19}, we also analyze the (time) complexity of the algorithm concretely.
In our framework, we first determine $(A_1, B_1, A_2, B_2)$ from a pair of supersingular $j$-invariants since a supersingular elliptic curve always exists for arbitrary $p>0$.
Moreover, we may assume $\nu = 1$, see Remark \ref{rem:nu1}.
The algorithm computes $(\lambda, \mu)$ as zeros over $\mathbb{F}_{p^2}$ of a multivariate system derived from the superspeciality of $C$.
Since the number of variables is two ($\lambda$ and $\mu$), computing $(\lambda, \mu)$ has the following three variants:
\begin{enumerate}
\item[(i)] Use the brute-force on $(\lambda, \mu) \in (\mathbb{F}_{p^2})^2$.
\item[(ii)] Conduct the brute-force only on $\lambda$ (resp.\ $\mu$), and compute $\mu$ (resp.\ $\lambda$) as a root of a univariate polynomial with the gcd computation.
\item[(iii)] Regarding both $\lambda$ and $\mu$ as variables, use an approach based on resultants.
\end{enumerate}
Analyzing the complexity of each variant, we shall determine the fastest one.

\paragraph{\bf (C) New isomorphism-tests for Howe curves.}
Once superspecial Howe curves have been listed by (B), we shall classify their isomorphism classes to obtain Theorem \ref{thm:main2}.
For this, an efficient isomorphism-test is required since many superspecial Howe curves are found by (B).
As any Howe curve is canonical (cf.\;Proposition \ref{HoweIsCanonical}), one may check whether two Howe curves are isomorphic, by the isomorphism-test given in \cite[6.1]{KH17} for canonical curves, whose implementation is found in \cite[4.3]{KH18}, see Subsection \ref{IsomTestCanModel} for a brief review.
However, it turns out to be very costly, since it requires much Gr\"obner basis computations.
As an alternative method, we propose the use of the relative positions of the ramified points of $E_i\to \bbP^1$ for a Howe curve $H$, which is the normalization of $E_1\times_{\bbP^1}E_2$, where two elliptic quotients $E_1,E_2$ of $H$ are not uniquely determined by $H$.
Our isomorphism-test (cf.\;Proposition \ref{prop:isom}) is obtained by enumerating all quotients
$E_1,E_2$ of $H$ such that $H$ is the normalizaion of $E_1\times_{\bbP^1}E_2$ and by
looking at the relative position of ramified points  $E_i\to \bbP^1$.
Computational results show that the classification of isomorphism classes by this new isomorphism-test can be much (about $100$ times on average) faster than that by the conventional one (\cite{KH17}, \cite{KH18}) for canonical curves.

\paragraph{\bf Organization.}
The organization of this paper is as follows.
In Section \ref{sec:Howe}, we start with recalling basic facts on Howe curves and explain the superspeciality-test for Howe curves.
In Section \ref{sec:IsomTest}, we introduce a new method to classify isomorphism classes of superspecial Howe curves, after we describe a way to enumerate certain elliptic quotients of a given Howe curve.
In Section \ref{sec:algorithm}, we give the two main algorithms and state the main results.
The details of the algorithms and some complexity analyses are given in Section \ref{sec:details}.
Section \ref{sec:conc} is a concluding remark.


%% file: 2and3.tex
\section{Howe curves and their superspeciality}\label{sec:Howe}
In this section, we recall the definition of Howe curves
and properties of these curves, and study
the superspeciality of them.
In particular, we show in Proposition \ref{prop:field_of_def} that the field of definition of superspecial Howe curves is $\mathbb{F}_{p^2}$.

\subsection{Definition and a realization}\label{subsec:def}
A {\it Howe curve} is a curve which is isomorphic to the desingularization of
the fiber product $E_1 \times_{{\bbP}^1} E_2$ of two double covers
$E_i\to {\bbP}^1$ ramified over $S_i$, where
$S_i$ consists of 4 points and $|S_1\cap S_2|=1$ holds.

We assume $p>3$, since there is no superspecial curve of genus $4$
if $p\le 3$ (cf.\;\cite[Theorem 1.1]{Ekedahl}).
The following realization of Howe curves seems to work well
in both of theoretical and computational viewpoints.
Let $K$ be an algebraically closed field in characteristic $p$.
Let $y^2 = x^3 + A_i x + B_i$ ($i=1,2$)
be two (nonsingular) elliptic curves, where $A_1,B_1,A_2,B_2\in K$.
Let $\lambda$, $\mu$ and $\nu$ be elements of $K$ and set
\begin{eqnarray}
f_1(x) &=& x^3 + A_1 \mu^2 x + B_1\mu^3,\label{f1}\\
f_2(x) &=& (x-\lambda)^3 + A_2 \nu^2(x-\lambda) + B_2 \nu^3.\label{f2}
\end{eqnarray}
The two elliptic curves
\begin{eqnarray}
E_1:& & z^2y = f_1^{(\rm h)}(x,y) := x^3 + A_1 \mu^2 xy^2 + B_1\mu^3y^3,\\
E_2:& & w^2y = f_2^{(\rm h)}(x,y) 
:=(x-\lambda y)^3 + A_2 \nu^2(x-\lambda y)y^2 + B_2 \nu^3y^3
\end{eqnarray}
are considered as  the double covers 
\begin{equation*}
\pi_i : E_i \to {\bbP}^1=\Proj K[x,y].
\end{equation*}
Note that the isomorphism classes of $E_1$ and $E_2$ are independent of the choice of $(\lambda,\mu,\nu)$ provided $\mu\ne 0$ and $\nu\ne 0$.
We say that $(\lambda,\mu,\nu)$ is {\it of Howe type} if
\begin{enumerate}
\item[(i)] $\mu\ne 0$ and $\nu\ne 0$;
\item[(ii)] $f_1$ and $f_2$ are coprime.
\end{enumerate}
If $(\lambda,\mu,\nu)$ is of Howe type, then
the desingularization $H$ of the fiber product $E_1\times_{{\bbP}^1}E_2$
is a Howe curve, since $E_i\to{\bbP}^1$ is ramified over the set consisting of 4 points, say $S_i$,
and $S_1\cap S_2 = \{(1:0)\}$.

\subsection{Howe curves are canonical}\label{subsec:can}
We show that any Howe curve is a canonical curve of genus $4$.
Note that a curve of genus $4$ is either hyperelliptic or canonical.

\begin{proposition}\label{HoweIsCanonical}
A Howe curve is a canonical curve of genus $4$.
\end{proposition}
\begin{proof}
Let $H$ be a Howe curve, i.e.,
the normalization of the fiber product over $\bbP^1$ of $E_1$ and $E_2$ as above.
Let $H'$ be the curve defined in $\bbP^3=\Proj k[x,y,z,w]$ by
\begin{eqnarray}
z^2-w^2 &=& q(x,y),\label{Howe_quadratic_form}\\
z^2y &=& x^3 + A_1 \mu^2 xy^2 + B_1\mu^3y^3\label{Howe_cubic_form},
\end{eqnarray}
where $q(x,y)$ is the quadratic form
\begin{equation}\label{eq:q}
q(x,y) = (f_1^{(\rm h)}(x,y)-f_2^{(\rm h)}(x,y))/y.
\end{equation}
Note that $H'$ and $E_1\times_{\bbP^1} E_2$ are isomorphic if
the locus $y=0$ is excluded.
It is straightforward to see that $H'$ is nonsingular
(if $(\lambda,\mu,\nu)$ is of Howe type).
Hence $H$ and $H'$ are isomorphic (cf.\,\cite[Chap.\,II, Prop.\,2.1]{Sil}).

It is well-known that any nonsingular curve defined by
a quadratic form and a cubic form in $\bbP^3$
is a canonical curve of genus $4$ (cf.\,\cite[Chap.\,IV, Exam.\,5.2.2]{Har}).
\end{proof}

\subsection{Superspeciality}\label{subsec:ssp}
Suppose that $(\lambda,\mu,\nu)$ is of Howe type.
Put
\begin{equation}
f(x) = f_1(x)f_2(x)
\end{equation}
and consider the hyperelliptic curve $C$ of genus $2$ defined by
\begin{equation*}
C: u^2 = f(x).
\end{equation*}
In \cite[Theorem 2.1]{Howe}, Howe proved  that there exist two isogenies
\begin{eqnarray*}
\varphi: J(H) \longrightarrow E_1 \times E_2 \times J(C),\\
\psi: E_1 \times E_2 \times J(C) \longrightarrow J(H)
\end{eqnarray*}
such that $\varphi\circ\psi$ and $\psi\circ\varphi$ are the multiplication by $2$, see \cite[Theorem C]{KR} for a result in more general cases.
If $p$ is odd, then
$\varphi\circ\psi$ and $\psi\circ\varphi$ are isomorphisms between the $p$-kernels of $J(H)$ and $E_1 \times E_2 \times J(C)$, whence
$J(H)[p]$ and $E_1[p] \times E_2[p] \times J(C)[p]$ are isomorphic, see \cite[Corollary 2]{GP} for a more general result.
Hence $H$ is superspecial if and only if $E_1$ and $E_2$ are supersingular and $C$ is superspecial. 

Now we recall a criterion for the superspeciality of $C$.
Let $\gamma_i$ be the $x^i$-coefficient of $f(x)^{(p-1)/2}$.
Put
\[
a=\gamma_{p-1},\quad
b=\gamma_{2p-1},\quad
c=\gamma_{p-2}\quad \text{and}\quad
d=\gamma_{2p-2}.
\]
Note that $\gamma_i$ and especially $a,b,c$ and $d$
are homogeneous as polynomials in $\lambda$, $\mu$ and $\nu$. 
Let $M$ be the Cartier-Manin matrix of $C$,
that is a matrix representing the Cartier operator
on $H^0(C,\Omega^1)$.
It is known (cf.\ \cite[4.1]{Gonzalez} and \cite[\S 2]{Yui}) that the Cartier-Manin matrix of $C$ is given by
\begin{equation}\label{Cartier-Manin matrix}
M := \begin{pmatrix}
a & b \\
c & d
\end{pmatrix}.
\end{equation}
\begin{lemma}\label{lem:ssp}
Let $H$ be a Howe curve as above.
Then $H$ is superspecial if and only if $E_1$ and $E_2$ are supersingular and $a=b=c=d=0$.
\end{lemma}

The next lemma will be used in Subsection \ref{subsec:complexity} to estimate the computational complexity of our algorithm. The proof will be given in the full paper.
\begin{lemma}\label{lem:deg}
If $E_1$ and $E_2$ are supersingular, then
\begin{itemize}
\item[(1)] $\mathrm{deg}_{\lambda} (a)$, $\mathrm{deg}_{\mu} (a)$ and $\mathrm{deg}_{\nu} (a)$ are at most $(3p-5)/2$;
\item[(2)] $\mathrm{deg}_{\lambda} (b) \le (p-5)/2$, and
$\mathrm{deg}_{\mu} (b)$ and $\mathrm{deg}_{\nu} (b)$ are at most $p-2$;
\item[(3)] $\mathrm{deg}_{\lambda} (c)$, $\mathrm{deg}_{\mu} (c)$ and $\mathrm{deg}_{\nu} (c)$ are at most $(3p-3)/2$;
\item[(4)] $\mathrm{deg}_{\lambda} (d)\le (p-3)/2$, and $\mathrm{deg}_{\mu} (d)$ and $\mathrm{deg}_{\nu} (d)$ are at most $p-1$.
\end{itemize}
\end{lemma}

Finally we discuss the field of definition of superspecial Howe curves.
\begin{proposition}\label{prop:field_of_def}
Let $K$ be an algebraically closed field in characteristic $p > 0$.
Any superspecial Howe curve $K$ is isomorphic to
$H$ obtained as above for $A_1$, $B_1$, $A_2$, $B_2$, $\mu$, $\nu$ and $\lambda$ belonging to $\F_{p^2}$.
\end{proposition}
\begin{proof}
It suffices to consider the case of $K=\overline{\F_{p^2}}$,
since any supersingular elliptic curve is defined over $\F_{p^2}$
and $(\lambda,\mu,\nu)$ is a solution of $a=b=c=d=0$.
Let $H'$ be a superspecial Howe curve over $K=\overline{\F_{p^2}}$.
Choose $E'_1$ and $E'_2$ over $K$ so that $H'$ is the normalization of
$E'_1\times_{\bbP^1} E'_2$.
It is well-known that $H'$ descends to $H$ over $\F_{p^2}$
such that the Frobenius map $F$ (the $p^2$-power map) on $\Jac(H)$ is $p$ or $-p$ and all automorphisms of $H$ are defined over $\F_{p^2}$ (cf.\;the proof of \cite[Theorem 1.1]{Ekedahl}). Let $E_1$ and $E_2$ be the quotients of $H$ corresponding to $E'_1$ and $E'_2$. The quotient $E_i$ of $H$ is obtained by an involution $\iota_i \in\Aut(H)$, and therefore is defined over $\F_{p^2}$.
The quotient by the group generated by $\iota_1$ and $\iota_2$
is isomorphic to $\bbP^1$ over ${\F_{p^2}}$.
Let $S_i$ be the set of the ramified points of $E_i \to \bbP^1$.
Since $S_1\cap S_2$ consists of a single point, the point is an invariant under the action by the absolute Galois group of $\F_{p^2}$ and therefore is an $\F_{p^2}$-rational point. An element of $\PGL_2(\F_{p^2})$ sends the point to the infinity $\in\bbP^1$.
Since the Frobenius map $F$ on $E_i$ is also $\pm p$,
any remaining element $P$ of $S_i$ (a $2$-torsion point on $E_i$)
are also $\F_{p^2}$-rational by $F(P)=\pm p P = P$.
This implies the desired result.
\end{proof}

\section{Isomorphism-tests for Howe curves}\label{sec:IsomTest}
In this section, we present two methods to determine whether two Howe curves are isomorphic or not.
The first method is given in \cite{KH17}, and it works for canonical curves.
We recall this in Subsection \ref{IsomTestCanModel}.
In the remaining subsections,
we introduce an isomorphism-test using ramified points
of the morphisms from elliptic curves to $\bbP^1$.
For this, to each Howe curve $H$ we associate a set ${\rm EQ}(H)$ of elliptic curves $E$
possibly appearing as $E_i$ so that $H$ is the normalization of $E_1\times_{\bbP^1}E_2$.

\subsection{Isomorphism-test as canonical curves}\label{IsomTestCanModel}
Let $k$ be an algebraically closed field in characteristic $p > 0$.
As shown in Proposition \ref{HoweIsCanonical}, any Howe curve is canonical.
Here, let us give a short review of the isomorphism-test in \cite[6.1]{KH17} for canonical curves. 

Recall that any canonical curve $X$ of genus $4$ is a closed subvariety $V(Q,P)$ in $\bbP^3=\Proj k[x,y,z,w]$ defined by an irreducible cubic form $P$ and an irreducible quadratic form $Q$ (cf.\ \eqref{Howe_quadratic_form} and \eqref{Howe_cubic_form} for a Howe curve).
It is known that $Q$ is uniquely determined by $X$ up to linear coordinate-change.
Since any irreducible quadratic form $Q$ over $k$ is isomorphic to either of (Non-Dege) $xw+yz$ and (Dege) $2yw+z^2$, we assume that the quadratic form is either of these from now on.

Let $V_1:=V(Q_1,P_1)$ and $V_2:=V(Q_2,P_2)$ be two canonical curves.
If $V_1$ and $V_2$ are isomorphic, then $Q_1=Q_2$.
Note that $V_1$ and $V_2$ are isomorphic if and only if an element $g$ of the orthogonal similitude group of $Q:=Q_1=Q_2$ satisfies $g P_1 \equiv \lambda P_2 \;\modulo\; Q$ for some $\lambda\in k^\times$.
An algorithm implemented by using a Bruhat decomposition of the orthogonal similitude group and Gr\"obner basis computation is found in \cite[4.3]{KH18}.

\subsection{Elliptic quotients of a canonical curve of genus $4$}\label{subsec:defEQ}
Let $k$ be an algebraically closed field in characteristic $p > 0$.
Let $H$ be a canonical curve of genus $4$ (not necessarily a Howe curve).
In this subsection, we classify certain dominant morphisms from $H$ to a genus-one curve of degree $2$.

Let $I$ be the ideal of $k[x,y,z,w]$ defining $H$, which is generated by a cubic form $P$ and a quadratic form $Q$.
Let $V$ be the 4-dimensional $k$-vector space generated by $x$, $y$, $z$ and $w$.
Consider a 3-dimensional subspace $U$ of $V$.
Let $\{X, Y, Z, W\}$ be a basis of $V$ such that $\{X, Y, Z\}$ is a basis of $U$.
Write $k[U]:=k[X,Y,Z]$.
If $I \cap k[U]$ is generated by a (single nonsingular) cubic form, i.e., there exist a genus-one curve $E_U=\Proj k[U]/I\cap k[U]$ and a morphism
\[
H \to E_U
\]
of degree $2$.
Note the cubic form in $I\cap k[U]$ is of the form $P-L\cdot Q$ up to scalar multiplication, where $L$ is a linear combination of $x$, $y$, $z$ and $w$.
We set
\[
{\rm EQ}(H) := \{U \subset V \mid \dim(U) =3, I\cap k[U]=(R) \text{ with } R \text{ is a cubic form}\} ,
\]
where $(R)$ denotes the ideal of $I \cap k[U]$ generated by $R$.
Here is a way to determine ${\rm EQ}(H)$ concretely.
Since $U$ is determined in $V$ by a single linear equation
\[
a_1 \xi_1+a_2 \xi_2+a_3 \xi_3+a_4 \xi_4=0
\]
for $\xi_1 x + \xi_2 y + \xi_3 z + \xi_4 w \in V$.
We write $U=V^{\perp {\bf a}}$ for ${\bf a} = (a_1,a_2,a_3,a_4)$.
\begin{enumerate}
\item[{\bf (A)}]
If $a_4 \ne 0$, then we may assume $a_4=1$.
As a basis of $V$, put $X=x-a_1w$, $Y=y-a_2w$, $Z = z-a_3 w$ and $W=w$,
where $X$, $Y$ and $Z$ span $U$.

\item[{\bf (B)}]
If $a_4 = 0$ and $a_3\ne 0$, then we may assume $a_3=1$.
As a basis of $V$, put $X=x-a_1z$, $Y=y-a_2z$, $Z=w$ and $W=z$,
where $X$, $Y$ and $Z$ span $U$.

\item[{\bf (C)}]
If $a_4 = a_3=0$ and $a_2\ne 0$, then we may assume $a_2=1$.
As a basis of $V$, put $X=x-a_1y$, $Y=z$, $Z=w$ and $W=y$,
where $X$, $Y$ and $Z$ span $U$.
\item[{\bf (D)}]
If $a_4 = a_3=a_2=0$ and $a_1\ne 0$, then we may assume $a_1=1$.
As a basis of $V$, put $X=y$, $Y=z$, $Z=w$ and $W=x$,
where $X$, $Y$, and $Z$ span $U$.
\end{enumerate}
We consider
\[
R:=P-(b_1 x+b_2 y+b_3 z+b_4 w)Q
\]
as a polynomial in $X$, $Y$, $Z$ and $W$.
If the coefficients of $R$ of the $10$ monomials containing $W$ (i.e., the coefficients of $X^2W$, $XYW$, $XZW$, $XW^2$, $Y^2W$, $YZW$, $YW^2$, $Z^2W$, $ZW^2$ and $W^3$) are all zero, then $R \in I\cap k[U]$.
By solving the 10 equations in (at most seven) variables $a_1$, $a_2$, $a_3$, $b_1$, $b_2$, $b_3$ and $b_4$, we get all $R$'s together with the change-of-basis matrices $\boldsymbol{U}_{\mathbf{a},\mathbf{b}}$ from $\{ x, y, z, w \}$ to $\{X, Y, Z, W \}$, i.e., matrices in $\mathrm{GL}_4 (k)$ satisfying $ (X, Y, Z, W) = (x, y, z, w) \cdot {}^t \boldsymbol{U}_{\mathbf{a},\mathbf{b}}$ for $\mathbf{a}:= (a_1, a_2, a_3, a_4)$ and $\mathbf{b}:= (b_1, b_2, b_3, b_4)$.
(In our computation, we represent each element $U \in \mathrm{EQ}(H)$ by such an $\boldsymbol{U}_{\mathbf{a},\mathbf{b}}$, see Algorithms \ref{alg:main2} and \ref{alg:sub1} below.)

\begin{remark}\label{rem:EQ}
For each element $U\in {\rm EQ}(H)$, we have a morphism $H \to E_U$ of degree $2$, which corresponds to an involution in $\Aut(H)$.
The cardinality of $\Aut(H)$ is finite and moreover is bounded by a constant independent of characteristic and $H$.
This is a well-known fact for curves of genus $\ge 2$.
Thus we have
\[
\#{\rm EQ}(H) \le \#\{\iota\in\Aut(H) \mid \iota^2={\rm id}_H\} < \infty.
\]
The cardinality of ${\rm EQ}(H)$ is an invariant of canonical curves of genus $4$.
It would be interesting to know the cardinalities of ${\rm EQ}(H)$ and $\{\iota\in\Aut(H) \mid \iota^2={\rm id}_H\}$.
An arithmetical method using Mass formula, etc.\ may be applicable if $H$ is superspecial.
See \cite{KHS} for an algorithm determining $\Aut(H)$.
\end{remark}
\begin{example}
Let $H$ be the unique superspecial curve of genus $4$ in characteristic $5$ (cf.\ \cite{FGT} and \cite{KH17}), which turns out to be a Howe curve.
A computation (with Magma~\cite{Magma}) says that $\#{\rm EQ}(H) = 10$ and  $\#\{\iota\in\Aut(H) \mid \iota^2={\rm id}_H\}=26$.
\end{example}

\subsection{A new isomorphism-test}\label{subsec:newIsom}
We propose a new method determining whether two Howe curves are isomorphic, where we use the ramified points of $E_i\to \bbP^1$ for a Howe curve $H$ realized as the normalization of $E_1\times_{\bbP^1} E_2$.

Let $U^{(1)}$ and $U^{(2)}$ be distinct elements of ${\rm EQ}(H)$,
say $U^{(i)}=V^{\perp {\bf a}^{(i)}}$ for ${\bf a}^{(i)} = (a_1^{(i)},a_2^{(i)},a_3^{(i)},a_4^{(i)})$.
Let $\{ X^{(i)},Y^{(i)},Z^{(i)},W^{(i)} \}$ be a basis of $V$ so that $\{X^{(i)},Y^{(i)},Z^{(i)} \}$ is a basis of $U^{(i)}$.
Let $R^{(i)}$ be the cubic obtained as above for $U^{(i)}$:
\[
R^{(i)} = P-L^{(i)}\cdot Q
\]
with
\[
L^{(i)} = b_1^{(i)}x+b_2^{(i)}y+b_3^{(i)}z+b_4^{(i)}w.
\]
Set $E^{(i)} = \Proj(k[U^{(i)}]/(R^{(i)}))$
and $\bbP^1=\Proj k[U^{(1)}\cap U^{(2)}]$.
If the natural morphisms $E^{(i)} \to \bbP^1$ are of degree $2$,
we say that $(U^{(1)},U^{(2)})$ is {\it admissible}.
Then $H$ is the normalization of $E^{(1)}\times_{\bbP^1}E^{(2)}$.

Assume that $(U^{(1)},U^{(2)})$ is admissible.
Put
\[
\underline{y}=L^{(1)}-L^{(2)}.
\]
\begin{lemma}\label{lem:y}
$\underline{y}\in U^{(1)}\cap U^{(2)}$.
\end{lemma}

\begin{proof}
Let $\{T_1,T_2 \}$ be a basis of $U^{(1)}\cap U^{(2)}$ and choose $\{ T_3^{(1)}, T_3^{(2)} \}$ such that $\{T_1,T_2,T_3^{(i)}\}$ be a basis of $U^{(i)}$.
Then $\{T_1,T_2,T_3^{(1)}, T_3^{(2)}\}$ is a basis of $V$.
It follows from $Q\not\in k[U^{(1)}]\cap I$ that $Q$ contains $T_3^{(2)}$.
Moreover since the morphism $H\to E^{(i)}$ is of degree $2$, the quadratic $Q$ contains the monomial $(T_3^{(2)})^2$.
Similarly $Q$ contains the monomial $(T_3^{(1)})^2$.
Also, since $E^{(i)} \to \Proj k[U^{(1)}\cap U^{(2)}]$ is of degree $2$, the cubic $R^{(i)}$ are of degree $2$ as a polynomial in $T_3^{(i)}$.
Then, by
\[
R^{(2)}-R^{(1)}=(L^{(1)}-L^{(2)})Q,
\]
we conclude that $L^{(1)}-L^{(2)}$ does not contain $T_3^{(1)}$ and $T_3^{(2)}$.
\end{proof}

Choose an element $\underline x$ of
\[
U^{(1)}\cap U^{(2)} = \left\{\xi_1x+\xi_2 y+ \xi_3 z + \xi_4 w \ \left|\ \sum_{k=1}^4 a_k^{(i)}\xi_k=0 \text{ for } i=1,2 \right. \right\}
\]
such that $\underline x$ and $\underline y$ are linearly independent.
Choose $\underline z^{(i)} \in \{X^{(i)}, Y^{(i)}, Z^{(i)}\}$ such that $\underline z^{(i)}\not\in U^{(1)}\cap U^{(2)}$.
Note that $\underline x$, $\underline y$, $\underline z^{(i)}$ and $W^{(i)}$ are linear combinations of $x$, $y$, $z$ and $w$, and they are linearly independent.
By the coordinate change 
to
$\{\underline x,\underline y,\underline z^{(i)}, W^{(i)}\}$,
we regard $R^{(i)}$ as a polynomial in $\underline x$, $\underline y$ and $\underline z^{(i)}$, since $W^{(i)}$ does not appear.
Substitute $1$ for $\underline{y}$ and consider $R^{(i)}$ as an element of $k[\underline x][\underline z^{(i)}]$.
Note that $(U^{(1)},U^{(2)})$ is admissible if and only if $R^{(i)}$ ($i=1,2$) are quadratic in $\underline z^{(i)}$.
%
Let $D^{(i)}$ be the discriminant of the quadratic polynomial $R^{(i)}$, which is a polynomial $\in k[\underline x]$ of degree $3$.
Let $\underline x_1^{(i)}$, $\underline x_2^{(i)}$ and $\underline x_3^{(i)}$ $(i=1,2)$ be the roots of $D^{(i)}$.
Then the morphism $E^{(i)}\to\bbP^1=\Proj k[\underline x, \underline y]$ is ramified at $(\underline x^{(i)}_k:1)$ ($k=1,2,3$) and the infinity $(1:0)$.
Write $v_H(U^{(1)},U^{(2)}) = (\underline x_1^{(1)},\underline x_2^{(1)},\underline x_3^{(1)},\underline x_1^{(2)},\underline x_2^{(2)},\underline x_3^{(2)})$.
As the isomorphism class of $H$ is determined by $v_H(U^{(1)},U^{(2)})$ up to arrangement of $\{\{\underline x_1^{(i)},\underline x_2^{(i)},\underline x_3^{(i)}\}\mid i=1,2\}$ and affine transformation (automorphism of $\bbP^1$ stabilizing the infinity), we have the following key proposition:

\begin{proposition}\label{prop:isom}
Let $H_0$ and $H$ be two Howe curves.
Let $(U_0^{(1)}, U_0^{(2)})$ be an admissible pair of elements of ${\rm EQ}(H_0)$.Write $v_{H_0}(U_0^{(1)},U_0^{(2)})=(\underline t_1^{(1)},\underline t_2^{(1)},\underline t_3^{(1)},\underline t_1^{(2)},\underline t_2^{(2)},\underline t_3^{(2)})$.
We have that $H_0$ and $H$ are isomorphic if and only if
there exists
an admissible pair $(U^{(1)}, U^{(2)})$ of elements of ${\rm EQ}(H)$
such that writting $v_H(U^{(1)},U^{(2)}) = (\underline x_1^{(1)},\underline x_2^{(1)},\underline x_3^{(1)},\underline x_1^{(2)},\underline x_2^{(2)},\underline x_3^{(2)})$, there exists $(\sigma,\tau,w)\in {\frak S}_3\times {\frak S}_3 \times {\frak S}_2$ such that
\begin{eqnarray}\label{eq:mat}
\rank\begin{pmatrix}
\underline x_{\sigma(1)}^{(w(1))}&\underline x_{\sigma(2)}^{(w(1))}&\underline x_{\sigma(3)}^{(w(1))}&\underline x_{\tau(1)}^{(w(2))}&\underline x_{\tau(2)}^{(w(2))}&\underline x_{\tau(3)}^{(w(2))}\\
\underline t_1^{(1)}&\underline t_2^{(1)}&\underline t_3^{(1)}&\underline t_1^{(2)}&\underline t_2^{(2)}&\underline t_3^{(2)}\\
1&1&1&1&1&1
\end{pmatrix}\le 2,
\end{eqnarray}
where $\frak S_n$ is the symmetric group of degree $n$.
\end{proposition}

%% file: 4.tex
\section{Main algorithms and main results}\label{sec:algorithm}

\subsection{Main algorithms}\label{subsec:algorithm}
Based on the theory described in the previous sections, we present the two main algorithms below.
Given a rational prime $p > 3$, the first one (Algorithm \ref{alg:main1}) lists superspecial Howe curves, and in particular it determines the (non-)existence of such a curve.

\begin{algorithm}\label{alg:main1}
\textsc{ListOfSuperspecialHoweCurves}($p$)
\begin{description}
\item[Input] A rational prime $p > 3$.
\item[Output] A list $\mathcal{H}(p)$ of superspecial Howe curves, each of which is represented by a tuple $(A_1, B_1, A_2, B_2, \lambda , \mu, \nu ) \in (\mathbb{F}_{p^2})^7$ of parameters.
\end{description}
\begin{enumerate}
\item Compute a list $\mathcal{C}$ of all Weierstrass coefficients pairs $(A, B) \in (\mathbb{F}_{p^2})^2$ such that $E_{A,B} : y^2 = x^3 + A x + B$ is supersingular, together with the set $J \subset \mathbb{F}_{p^2}$ of all supersingular $j$-invariants as follows:
\begin{enumerate}
\item[(1-1)] Set $J$ $\leftarrow$ $\emptyset$, $\mathcal{C}$ $\leftarrow$ $\emptyset$ and $e$ $\leftarrow$ $(p-1)/2$.
\item[(1-2)] Compute the set $\mathcal{T}$ of all roots of $H_p(t)=\sum_{i=0}^e \binom{e}{i}^2 t^i$. 
\item[(1-3)] For each $t_0 \in \mathcal{T}$:
\begin{enumerate}
\item[(1-3-1)] Compute $j_0 := 2^8 (t_0^2 - t_0 + 1)^3 t_0^{-2} (t_0-1)^{-2}$.
\item[(1-3-2)] If $j_0 \notin J$, set $J$ $\leftarrow$ $J \cup \{ j_0 \}$ and $\mathcal{C}$ $\leftarrow$ $\mathcal{C} \cup \{ (A_0, B_0) \}$, where $A_0:=- (t_0^2 - t_0 + 1)/3$ and $B_0:=- (2 t_0^3 - 3 t_0^2 - 3 t_0 + 2)/27$.
\end{enumerate}
{\bf Note:} Each $j_0$ is equal to the $j$-invariant of the supersingular elliptic curve $E_{t_0}: y^2 = x (x-1)(x-t_0)$ in Legendre form, and $E_{A_0,B_0}$ is a supersingular elliptic curve over $\mathbb{F}_{p^2}$ with $j$-invariant equal to $j_0$.
\end{enumerate}
\item Set $\mathcal{H}(p)$ $\leftarrow$ $\emptyset$.
For each pair of $(A_1, B_1)$ and $(A_2, B_2)$ in $\mathcal{C}$, possibly choosing $(A_1,B_1) = (A_2,B_2)$, compute all $(\lambda , \mu, \nu ) \in (\mathbb{F}_{p^2})^3$ of Howe type such that the normalization $H$ of $E_1 \times_{\mathbf{P}^1} E_2$ is superspecial, where $E_i$ is an elliptic curve defined in Subsection \ref{subsec:def} with $E_i \cong E_{A_i, B_i}$ over $\mathbb{F}_{p^2}$ for each $i$.
\begin{enumerate}
\item[(2-1)] Compute the set $\mathcal{V}(A_1,B_1,A_2,B_2)$ of elements $(\lambda, \mu, \nu)$ such that the Cartier-Manin matrix \eqref{Cartier-Manin matrix} is zero.
\item[(2-2)] For each $(\lambda, \mu, \nu) \in \mathcal{V}(A_1,B_1,A_2,B_2)$:
Test whether $f_1$ and $f_2$ in \eqref{f1} and \eqref{f2} are coprime or not, by computing the resultant $\mathrm{Res}_x(f_1,f_2)$.
If $f_1$ and $f_2$ are coprime, i.e., $\mathrm{Res}_x(f_1,f_2) \neq 0$, set $\mathcal{H}(p)$ $\leftarrow$ $\mathcal{H}(p) \cup \{ (A_1, B_1, A_2, B_2, \lambda, \mu, \nu) \}$.
\end{enumerate}
{\bf Note:} It suffices to compute elements $(\lambda, \mu, \nu)$ with $\nu = 1$, see Remark \ref{rem:nu1} below.
(The other cases (i)-(iii) are possible, this paper considers only the case (iv) for simplicity.)
\item Output $\mathcal{H}(p)$.
\end{enumerate}
\end{algorithm}

\begin{remark}\label{rem:nu1}
In Step (2) of Algorithm \ref{alg:main1}, it suffices to compute $(\lambda, \mu, \nu)$ satisfying one of the following four cases:
(i) $\lambda = 1$ or $(\lambda, \mu) = (0,1)$;
(ii) $\lambda = 1$ or $(\lambda, \nu) = (0,1)$;
(iii) $\mu = 1$; or
(iv) $\nu = 1$.
Namely, we may assume
\begin{enumerate}
\item[(i)] $\mathcal{V}(A_1,B_1,A_2,B_2) \subset \{ ( 1, \mu, \nu ) : (\mu , \nu) \in (\mathbb{F}_{p^2}^{\times})^2 \} \cup \{ ( 0, 1, \nu ) : \nu \in \mathbb{F}_{p^2}^{\times} \}$;
\item[(ii)] $\mathcal{V}(A_1,B_1,A_2,B_2) \subset \{ ( 1, \mu, \nu ) : (\mu , \nu) \in (\mathbb{F}_{p^2}^{\times})^2 \} \cup \{ ( 0, \mu, 1 ) : \mu \in \mathbb{F}_{p^2}^{\times} \}$;
\item[(iii)] $\mathcal{V}(A_1,B_1,A_2,B_2) \subset \{ ( \lambda, 1, \nu ) : (\lambda, \nu) \in \mathbb{F}_{p^2} \times \mathbb{F}_{p^2}^{\times} \}$; or
\item[(iv)] $\mathcal{V}(A_1,B_1,A_2,B_2) \subset \{ ( \lambda, \mu, 1 ) : (\lambda, \mu) \in \mathbb{F}_{p^2} \times \mathbb{F}_{p^2}^{\times} \}$.
\end{enumerate}
The reason we may assume this is that if $H_0$ and $H$ are Howe curves obtained respectively by $(A_1, B_1, A_2, B_2, \lambda, \mu, \nu)$ and $(A_1, B_1, A_2, B_2, c \lambda, c \mu, c \nu)$ for some $c \in k^{\times}$ with $k := \overline{\mathbb{F}_p}$, then $H_0 \cong H$ over $k$.
Indeed, the morphism sending $(x, y, z, w)$ to $(x, c^{-1} y, \sqrt{c} z, \sqrt{c} w)$ gives an isomorphism between $H_0$ and $H$, considering their canonical forms (cf.\ Subsection \ref{subsec:can}).
\end{remark}

Note that Step (1-2) (resp.\ (2-1)) has two (resp.\ three) variants with different complexities, see Subsection \ref{subsec:complexity} for more details and the fastest ones.

Once superspecial Howe curves are listed by Algorithm \ref{alg:main1}, the second main algorithm (Algorithm \ref{alg:main2}) classifies their isomorphism classes over $\overline{\mathbb{F}_p}$.

\begin{algorithm}\label{alg:main2}
\textsc{IsomorphismClassesOfHoweCurves}($\mathcal{H}(p)$)
\begin{description}
\item[Input] A list $\mathcal{H}(p)$ computed by Algorithm \ref{alg:main1}.
\item[Output] A list $\mathcal{H}_{\rm alc}(p) \subset \mathcal{H}(p)$ that represents the set of the $\overline{\mathbb{F}_p}$-isomorphism classes of Howe curves defined by $(A_1, B_1, A_2, B_2, \lambda , \mu, \nu ) \in \mathcal{H}(p)$.
\end{description}
\begin{enumerate}
\item For each $(A_1, B_1, A_2, B_2, \lambda , \mu, \nu ) \in \mathcal{H}(p)$:
\begin{enumerate}
\item[(1-1)] Represent the Howe curve $H$ defined by $(A_1, B_1, A_2, B_2, \lambda , \mu, \nu )$ as a canonical model $V(Q,P)$ for $Q:=z^2 - w^2 - q(x, y)$ and $P:= z^2 y - f_1(x)$, where $f_1(x)$ (resp.\ $q(x,y)$) is defined in \eqref{f1} (resp.\ \eqref{eq:q}).
\item[(1-2)] Compute the set $\mathrm{EQ}(H)$ by using the method described in Subsection \ref{subsec:defEQ} (for a pseudo code, see Algorithm \ref{alg:sub1} in Subsection \ref{subsec:EQ} below).
In our computation, we represent each element $U \in \mathrm{EQ}(H)$ by $(\mathbf{a}, \mathbf{b}, \boldsymbol{U}_{\mathbf{a},\mathbf{b}} )$, where $\mathbf{a}:= (a_1, a_2, a_3, a_4)$ and $\mathbf{b}:= (b_1, b_2, b_3, b_4)$ are vectors in $k^4$ with $k := \overline{\mathbb{F}_p}$ such that $a_1 \xi_1 + a_2 \xi_2 + a_3 \xi_3 + a_4 \xi_4 = 0$ for all $(\xi_1, \xi_2, \xi_3, \xi_4) \in k^4$ with $\xi_1 x + \xi_2 y + \xi_3 z + \xi_4 w \in V := k \langle x, y, z, w \rangle$, and where $\boldsymbol{U}_{\mathbf{a},\mathbf{b}} \in \mathrm{GL}_4(k)$ is a basis matrix for the $3$-dimensional space $U$.
\item[(1-3)] For each pair of $(\mathbf{a}^{(i)}, \mathbf{b}^{(i)}, \boldsymbol{U}_{\mathbf{a}^{(i)}, \mathbf{b}^{(i)}} )$ with $i = 1$ and $2$ representing two distinct elements $U_1$ and $U_2$ in $\mathrm{EQ}(H)$ respectively, compute $v_H (U^{(1)},U^{(2)})$ by the method described in Subsection \ref{subsec:newIsom} (for a pseudo code, see Algorithm \ref{alg:sub2} in Subsection \ref{subsec:VH} below).
Let $V_H$ be the set of computed $v_H (U^{(1)},U^{(2)})$'s.
\end{enumerate}
\item To obtain the desired set $\mathcal{H}_{\rm alc}(p)$, apply Proposition \ref{prop:isom} to determine $H_0 \cong H$ over $\overline{\mathbb{F}_p}$ or not from $V_{H_0}$ and $V_H$, for each pair of $H_0$ and $H$ in $\mathcal{H} (p)$.
\item Output $\mathcal{H}_{\rm alc}(p)$.
\end{enumerate}
\end{algorithm}


\subsection{Main results}\label{subsec:results}
This section presents results we have obtained by executing the main algorithms (Algorithms \ref{alg:main1} and \ref{alg:main2}) in Subsection \ref{subsec:algorithm}. 
We implemented the algorithms over Magma V2.22-3 in its 32-bit version on a laptop with Windows OS, a 2.60 GHz Inter Core i5-4210M processor and 8.00 GB memory.
Until this submission, we have executed Algorithm \ref{alg:main1} (resp.\ Algorithm \ref{alg:main2}) for $p$ up to $331$ (resp.\ $53$), and it took about $3.5$ hours (resp.\ $0.5$ day) in total to complete the execution.
Note that we stopped Algorithms \ref{alg:main1} for $p > 53$ once one $(A_1, B_1, A_2, B_2, \lambda, \mu, \nu)$ was found in Step (2).
As a result, we have the following:
\begin{itemize}
\item ({\bf Theorem \ref{thm:main1}}) There exists a superspecial Howe curve (and so a canonical curve of genus $4$) in characteristic $p$ for every prime $p=5$ or $7 < p \leq 331$.
\end{itemize}
For $p > 331$, an out of memory error happens, but we could avoid this error by using Magma on a PC with the other OS\footnote[1]{Note that Magma on Windows is only supported in 32-bit, and the memory limit for the 32-bit application is 1.30 GB.}, which we plan to adopt in the near future.
Also by the execution of Algorithm \ref{alg:main2}, 
\begin{itemize}
\item ({\bf Theorem \ref{thm:main2}}) We have determined the number of $\overline{\mathbb{F}_{p}}$-isomorphism classes of superspecial Howe curves for $5 \leq p \leq 53$ with no out of memory error.
\end{itemize}
Table \ref{table:1} below summarizes the results on the number of $\overline{\mathbb{F}_{p}}$-isomorphism classes, which we denote by $n (p)$.
For a comparison, we also implemented the conventional method in \cite{KH17} (for a brief review, see Subsection \ref{IsomTestCanModel} of this paper) to classify the isomorphism classes.
Note that Step (1-3) of Algorithm \ref{alg:main2} tests $H_0 \cong H$ from the inputs $V_{H_0}$ and $V_{H}$, which are pre-computed in Step (1-2).
On the other hand, the conventional method~\cite{KH17} decides $H_0 \cong H$ from (the standard forms of) $H_0$ and $H$.

\begin{table}[t]
\centering{
\caption{Our enumeration results obtained by Algorithm \ref{alg:main2} (and by the method in \cite{KH17} for a comparison on timing).
For each method, we denote by $t_{\rm Iso}$ the average time of isomorphism-tests.
}
\label{table:1}
\vspace{-2mm}
\begin{centering}
\scalebox{0.88}{
\begin{tabular}{c||c|c||r|r|r||r||r||r} \hline
& & $n(p)$ & \multicolumn{6}{|c}{Benchmark timing data (all times shown are in seconds)} \\ \cline{4-9}
$p$ & $\# \mathcal{H}(p)$ & {\bf (Thm.} & \multicolumn{4}{|c||}{Our new method (Algorithm \ref{alg:main2})} & \multicolumn{2}{|c}{The method~\cite{KH17}} \\ \cline{4-9}
&  & {\bf 1.3)} & \multicolumn{1}{|c|}{Step (1)} & \multicolumn{1}{|c|}{Step (2)} & \multicolumn{1}{|c||}{$t_{\rm Iso}$} & \multicolumn{1}{|c||}{{\bf total time}} & \multicolumn{1}{|c||}{$t_{\rm Iso}$} & \multicolumn{1}{|c}{{\bf total time}} \\ \hline
$5$ & $9$ & $1$ & $0.312$ & $<0.001$ & $<0.001$ & $0.312$ & $0.008$ & $0.075$ \\ \hline
$7$ & $0$ & $0$ & - & - & - & - & - & - \\ \hline
$11$ & $87$ & $4$ & $0.750$ & $2.078$ & $0.007$ & $2.828$  & $0.274$ & $45.159$ \\ \hline
$13$ & $126$ & $3$ & $0.406$ & $0.359$ & $< 0.001$ & $0.765$ & $0.441$ & $164.314$ \\ \hline
$17$ & $288$ & $10$ & $2.140$ & $20.564$ & $0.008$ & $22.704$ & $0.418$ & $681.770$ \\ \hline
$19$ & $174$ & $4$ & $1.265$ & $2.578$ & $0.004$ & $3.843$ & $0.287$ & $187.388$ \\ \hline
$23$ & $1089$ & $33$ & $4.734$ & $68.629$ & $0.002$ & $73.363$ &$0.552$ & $12603.039$ \\ \hline
$29$ & $1575$ & $45$ & $8.157$ & $221.404$ & $0.003$ & $229.561$ & $0.513$ & $22976.066$ \\ \hline
$31$ & $2166$ & $59$ & $9.373$ & $297.706$ & $0.002$ & $307.079$ & - &  \\ \cline{1-8}
$37$ & $1548$ & $41$ & $6.206$ & $139.207$ & $0.002$ & $145.413$ & - & \multicolumn{1}{|c}{In progress}  \\ \cline{1-8}
$41$ & $3720$ & $105$ & $10.980$ & $1473.208$ & $0.002$ & $1490.003$ & - & \multicolumn{1}{|c}{(more than}  \\ \cline{1-8}
$43$ & $3024$ & $79$ & $11.174$ & $708.858$ & $0.002$ & $720.032$ & - & \multicolumn{1}{|c}{$0.5$ day for} \\ \cline{1-8}
$47$ & $8843$ & $235$ & $34.886$ & $21763.204$ & $0.002$ & $21798.090$ & - & \multicolumn{1}{|c}{each $p$)} \\ \cline{1-8}
$53$ & $5949$ & $167$ & $23.932$ & $7555.510$ & $0.002$ & $7579.442$ & - &  \\ \hline
\end{tabular}
}
\end{centering}
}
\vspace{-3mm}
\end{table}

We see from Table \ref{table:1} that the dominant step for Algorithm \ref{alg:main2} is Step (2) for $5 \leq p \leq 53$ although Step (1-2) includes Gr\"{o}bner basis computations\footnote[2]{Both the methods compute Gr\"{o}bner bases, which we compute by the Magma's built-in function \texttt{GroebnerBasis}.}.
Note that once $V_H$'s are computed in Step (1), only simple linear algebra is used to conduct Step (2).
Another advantage of our classification algorithm is that Step (1) (the computation of $\mathrm{EQ}(H)$'s and $V_H$'s) can be parallelized perfectly and easily.

Here, we also give a brief comparison between the classification of isomorphism classes by Algorithm \ref{alg:main2} and that by the conventional method~\cite{KH17}.
For all primes $p >5$ shown in Table \ref{table:1}, we see from the columns on {\bf total time} that the former one is dramatically faster than the latter one (e.g., for $p = 13$, $23$, $29$ more than $100$ times faster).
In particular, the average time $t_{\rm Iso}$ of isomorphism-tests in Step (2) is $< 0. 01$ seconds, which is much less than that by the conventional one for any $p > 5$.
This is because the latter one computes Gr\"{o}bner bases for {\it each} isomorphism-test, whereas Step (2) of our method has no such a computation.
While we have not determine the complexity for any of Algorithm \ref{alg:main2} and the method~\cite{KH17} yet, from these practical behaviors we also expect that the complexities of Steps (1) and (2) of Algorithm \ref{alg:main2} are $O (\# \mathcal{H}(p) )$ and $O ( (\# \mathcal{H}(p) )^2 )$ (or $O ( n(p) )$ and $O ( n(p)^2)$) respectively, and that the complexity of Algorithm \ref{alg:main2} is $O ( (\# \mathcal{H}(p) )^2 )$ (or $O ( n(p)^2)$).


\begin{remark}\label{rem:asymptotic}
In Algorithm \ref{alg:main2}, the cost of Step (1) might be asymptotically more expensive than that of Step (2) since Step (1-2) requires Gr\"{o}bner basis computations.
However, the total number of Gr\"{o}bner basis computations required in Algorithm \ref{alg:main2} is always less than that of the classification by the method in \cite{KH17}; the former is $\# \mathcal{H}(p) = O(p^5)$ but the latter is $ \#  \mathcal{H}(p) ( \#  \mathcal{H}(p) + 1)/2 = O ( \# \mathcal{H}(p) )^2) = O(p^{10})$.
These complexities are not reduced even if one uses the divide-and-conquer method.
Thus, as long as the cost of each Gr\"{o}bner basis computation in Step (1) is less than that in the method~\cite{KH17}, Algorithm \ref{alg:main2} is (about $p^5$ times) more efficient than the classification by the method in \cite{KH17}.
\end{remark}

%% file: 5and6_v2.tex
\section{Algorithm details with complexity analysis}\label{sec:details}

In this section, we describe the details (of some steps) of our main algorithms.
In particular, we describe possible variants of Steps (1-2) and (2-1) in Algorithm \ref{alg:main1}, together with complexity analysis.
As for Algorithm \ref{alg:main2}, this section also presents pseudo codes of Steps (1-2) and (1-3), which are new computational methods proposed in Section \ref{sec:IsomTest}.

\subsection{Complexity analysis of Algorithm \ref{alg:main1} with variants of sub-routines}\label{subsec:complexity}

All time complexity bounds refer to arithmetic complexity, which is the number of operations in $\mathbb{F}_{p^2}$.
We denote by $\mathsf{M}(n)$ the time to multiply two univariate polynomials over $\mathbb{F}_{p^2}$ of degree $\leq n$.
We use the ``Soft-${\rm O}$'' notation to omit logarithmic factors; $g = O^{\sim} (h) $ means $g = O (h \log^k (h))$ for some constant $k$.
In the following, after we discuss possible variants of Steps (1-2) and (2-1) of Algorithm \ref{alg:main1}, we shall determine the total (time) complexity of Algorithm \ref{alg:main1} for the fastest variants.

\subsubsection{Variants of Step (1-2) of Algorithm \ref{alg:main1}}\label{subsec:1-2}
As we mentioned in Subsection \ref{subsec:algorithm}, there are two variants to compute the set $\mathcal{T}$ of all roots $t_0$ of $H_p (t)$.

The first one is the brute-force on $t_0 \in \mathbb{F}_{p^2}$ to check $H_p (t_0) = 0$ or not, where the evaluation $H_p (t)$ at $t = t_0$ is done in $O (p)$ since $\deg (H_p (t)) = e = O(p)$.
Thus, the complexity of this brute-force for Step (1-2) is $O (p^2) \times O (p) = O(p^3)$.

The second variant is to factor the univariate polynomial $H_p (t)$ over $\mathbb{F}_{p^2}$.
Its complexity is estimated as $O( {\log}^{2} (p) \mathsf{M}(p) )$, or $O^{\sim} (\mathsf{M}(p) ) $ if less precisely.
Indeed, since $H_{p}(t)$ is square-free and since it splits completely over $\mathbb{F}_{p^2}$, all of its roots are found only by the equal-degree factorization (EDF), e.g., \cite{GS}.
The EDF algorithm given in \cite{GS} completely factors a univariate polynomial of degree $D$ over $\mathbb{F}_q$ in $O({\log} (D) \log (q) \mathsf{M}(D) )$.
In our case with $D = e = O(p)$ and $q = p^2$, the complexity of EDF is $O ({\log}^{2} (p) \mathsf{M}(p) )$, as desired.

From this, the second variant is faster than the first one, and thus we adopt the second one in our implementation.

\subsubsection{Variants of Step (2-1) of Algorithm \ref{alg:main1}}\label{subsec:2-1}
As we mentioned in Subsection \ref{subsec:algorithm}, there are three variants (the following (i)-(iii)) to compute all $(\lambda, \mu, \nu) \in (\mathbb{F}_{q})^3$ with $\nu = 1$ and $q=p^2$ such that $M=0$, where $M$ is the Cartier-Manin matrix given in \eqref{Cartier-Manin matrix} with entries $a$, $b$, $c$ and $d$.

\begin{enumerate}
\item[(i)] Use the brute-force on $(\lambda, \mu) \in (\mathbb{F}_{p^2})^2$ to check $M =0$ or not.
\item[(ii)] Regard one of $\lambda$ and $\mu$ as a variable.
For simplicity, regard $\lambda$ as a variable.
For each $\mu \in \mathbb{F}_{p^2}$, compute the roots in $\mathbb{F}_{p^2}$ of $G:= \mathrm{gcd}(a, b, c, d) \in \mathbb{F}_{p^2}[\lambda ]$.
\item[(iii)] Regarding both $\lambda$ and $\mu$ as variables, use an approach based on resultants.
\end{enumerate}
In the following, we show that the fastest one is (ii).
In each of (i)-(iii), we compute the Cartier-Manin matrix $M$ from $f :=f_1 f_2$, where $f_1$ and $f_2$ are defined respectively in \eqref{f1} and \eqref{f2}.
As we will describe in Remark \ref{rem:compHW} below, the costs of computing $M$ for (i) and (ii) are bounded respectively by $O(p)$ and $O (p \mathsf{M}(p))$ $\mathbb{F}_{p^2}$ operations.
It is clear that the cost of computing $M$ for (iii) is not less than that for each of (i) and (ii).
For reasons of space, we give only a summary of the estimation of the complexities of (i), (ii) and (iii) (further
details to appear elsewhere).

The complexity of (i) is $p^4 \times O(p) = O(p^5)$ since the number of iterations is $p^4$.

For (ii), once $M$ is computed in $O (p \mathsf{M}(p))$, the dominant part of the rest of (ii)  is to factor $G$ with degree $\leq D$ over $\mathbb{F}_{q}$ with $q = p^2$, where $D$ denotes the maximum value of the degrees of $a$, $b$, $c$ and $d$ with respect to $\lambda$.
Factoring $G$ is done by the square-free, distinct-degree and equal-degree factorizations, and its cost is estimated as $O (D \log (D) \mathsf{M}(D) + \log (D) \log (q) \mathsf{M} (D) )$ by e.g., the method in \cite{GS}.
Since the number of iterations on $\mu$ is $q = p^2$, the complexity of (ii) is
\begin{eqnarray}
O (q p \mathsf{M}(p) + q D \log (D) \mathsf{M}(D) + q \log (D) \log (q) \mathsf{M} (D) ) .  \label{eq:comp2}
\end{eqnarray}
By Lemma \ref{lem:deg}, we have $D = O (p)$ and thus \eqref{eq:comp2} is $O (p^3 \log (p) \mathsf{M}(p) )$.
If one uses fastest algorithms with $ \mathsf{M}(n) = n \log (n) \log (\log (n))$ for polynomial multiplication, \eqref{eq:comp2} is $O (p^4 \log^2 (p) \log (\log (p) ))$.


In (iii), a way of computing the zeros $(\lambda, \mu)$ is the following:
First we compute the resultants $R_1 (\mu) = \mathrm{Res}_{\lambda}(a,b) \in \mathbb{F}_{q} [\mu]$ and $R_2 (\mu) =  \mathrm{Res}_{\lambda} (c,d) \in \mathbb{F}_{q} [\mu]$.
We next compute $G' := \mathrm{gcd}(R_1,R_2) $ and its roots over $\mathbb{F}_{q}$.
For each of computed roots, it suffices to compute the roots over $\mathbb{F}_{q}$ of $\mathrm{gcd}(a, b, c, d) \in \mathbb{F}_{q} [\lambda]$ for $a$, $b$, $c$ and $d$ evaluated at $\mu$.
With the total degree bound $O (D)$ on $a$, $b$, $c$ and $d$, we can estimate that the complexity of (iii) by the above way is the cost of computing $M$ plus
\begin{eqnarray}
O (D^2 \log (D) \mathsf{M}(D^2) + \log (D) \log (q) \mathsf{M} (D^2) ) . \label{eq:comp3}
\end{eqnarray}
If one ignores logarithmic factors in \eqref{eq:comp2} and \eqref{eq:comp3}, the lower one is \eqref{eq:comp2} (resp.\ \eqref{eq:comp3}) if $q < D^2$ (resp.\ $q > D^2$).
In our case, it follows from $D = O(p)$ and $q =p^2$ that \eqref{eq:comp2} and \eqref{eq:comp3} ignoring logarithmic factors give the same bound.
Thus, the complexity \eqref{eq:comp2} of (ii) does not exceed that of (iii), and (ii) is faster than (iii) if the cost of computing $M$ for (iii) is greater than \eqref{eq:comp3}.



From this, we adopt the fastest variant (ii) in our implementation, and assume that the complexity of Step (2-1) of Algorithm \ref{alg:main1} is bounded by $O (p^3 \log (p) \mathsf{M}(p))$.

\begin{remark}\label{rem:compHW}
In each of (i)-(iii), we require to compute the Cartier-Manin matrix.
In general, it is known that computing the Cartier-Manin matrix $M$ of a hyperelliptic curve $y^2 = f(x)$ defined over a field $K$ is reduced into multiplying matrices via recurrences by Bostan-Gaudry-Schost~\cite[Section 8]{BGS} for the coefficients of $f(x)^n$, see e.g., \cite[Section 2]{HS2} for details.
Harvey-Sutherland's algorithm \cite{HS2}, which is an improvement of the one~\cite{HS} presented at ANTS XI, is also based on this reduction, and it is the fastest algorithm to compute $M$ for the case of $K=\mathbb{F}_p$.
From this, one of the best ways to compute $M$ in (i) (resp.\ (ii), (iii)) is to extend Harvey-Sutherland's algorithm \cite{HS2} to the case of $K = \mathbb{F}_{p^2}$ (resp.\ $\mathbb{F}_{p^2}(\lambda)$, $\mathbb{F}_{p^2} (\lambda, \mu)$).
However, since we have not succeeded in the extension yet, we compute $M$ just by the reduction mentioned above.
In this case, computing $M$ in (i) (resp.\ (ii)) is done by the multiplication of matrices of $\mathbb{F}_{p^2}$ (resp.\ $\mathbb{F}_{p^2}[\lambda]$ together with constant times of divisions in $\mathbb{F}_{p^2}[\lambda]$), and its complexity is estimated as $O (p)$ (resp.\ $ O (p \mathsf{M}(p)) $).
The complexity for the case of (iii) has not been determined yet, but at least it is not less than those for the other cases. 
\end{remark}


\subsubsection{Total complexity of Algorithm \ref{alg:main1}}
Here we estimate the total complexity of Algorithm \ref{alg:main1}.
The complexity of Step (1) is dominated by that of Step (1-2), which we have estimated as $O ({\log}^{2} (p) \mathsf{M}(p) )$ in Subsection \ref{subsec:1-2}.
For Step (2), first recall from Subsection \ref{subsec:2-1} that the complexity of Step (2-1) is $O ( p^3 \log (p) \mathsf{M}(p) )$.
The number of $(\lambda, \mu, \nu)$ with $\nu = 1$ computed in Step (2-1) is $ \leq p^2 \times \mathrm{deg}(G)  = O (p^3)$.
For each $(\lambda, \mu, \nu)$, the number of operations in $\mathbb{F}_{p^2}$ for computing $\mathrm{Res}(f_1,f_2)$ does not depend on $p$, and thus the complexity of Step (2-2) is $O (p^3)$.
Since the number of possible choices of $\{ (A_1, B_1), (A_2, B_2) \}$ is $\# \mathcal{C} = O (p^2)$, computing $(\lambda, \mu, \nu)$ with $\nu = 1$ for all $\{ (A_1, B_1), (A_2, B_2) \}$ is done in $\# \mathcal{C} \times  O \left( p^3 \log (p) \mathsf{M}(p) \right) =   O \left( p^5 \log (p)  \mathsf{M}(p) \right)$ operations in $\mathbb{F}_{p^2}$.
Summing up the complexities of Steps (1) and (2), the total complexity of Algorithm \ref{alg:main1} is $O \left( p^5 \log (p)  \mathsf{M}(p) \right)$, or $O^{\sim} (p^6)$ if less precisely.

\subsection{Pseudo codes for some steps of Algorithm \ref{alg:main2} for implementation}

\subsubsection{Pseudo code for Step (1-2) of Algorithm \ref{alg:main2}}\label{subsec:EQ}
We use the same notation as in Subsections \ref{subsec:defEQ} and \ref{subsec:algorithm}.
Let $H:= V (Q, P) \subset \mathbf{P}^3$ be a canonical curve of genus $4$ (not necessarily Howe curve), where $Q$ (resp.\ $P$) is an irreducible quadratic (resp.\ cubic) form of $\mathbb{F}_{p^2} [x, y, z, w]$.
Given a pair $(Q, P)$, Algorithm \ref{alg:sub1} below computes a subset $\mathrm{EQ}(H)$ of $k^8 \times \mathrm{GL}_4 (k)$ with $k := \overline{\mathbb{F}_{p}}$ together with an integer $\ell_H \geq 1$ and a finite field $K_H$, where each element in $\mathrm{EQ}(H)$ is of the form $(\mathbf{a}, \mathbf{b}, \boldsymbol{U}_{\mathbf{a},\mathbf{b}} ) \subset (K_H)^8 \times \mathrm{GL}_4 (K_H)$ with $\mathbb{F}_{p^2} \subset K_H:= \mathbb{F}_{p^{2\ell_H}} \subset k$, see Step (1-2) of Algorithm \ref{alg:main2}.
\begin{algorithm}\label{alg:sub1}
Sub-algorithm 1: \textsc{ComputeEQ}($Q,P$)
\begin{description}
\item[Input] An irreducible quadratic form $Q$ and an irreducible cubic form $P$ of $\mathbb{F}_{p^2} [x,y,z,w]$.
\item[Output] A subset $\mathrm{EQ}(H)$ of $k^8 \times \mathrm{GL}_4 (k)$ with $k := \overline{\mathbb{F}_{p}}$, an integer $\ell_H \geq 1$, and a finite field $K_H$.
\end{description}
\begin{enumerate}
\item Set $\mathrm{EQ}(H)$ $\leftarrow$ $\emptyset$, $\ell$ $\leftarrow$ $1$ and $K_0$ $\leftarrow$ $\mathbb{F}_{p^2}$.
Let $X$, $Y$, $Z$, $W$, $a_1'$, $a_2'$, $a_3'$, $b_1'$, $b_2'$, $b_3'$ and $b_4'$ be indeterminates, where we use the notation $a_i^{\prime}$ and $b_j^{\prime}$ (resp.\ $a_i$ and $b_j$) for indeterminates (resp.\ exact elements in $k$) for distinction.
\item Compute $R :=  P - (b_1' x + b_2' y + b_3' z + b_4' w) Q \in S[x,y,z,w]$ with $S:=k[a_1', a_2', a_3', b_1', b_2', b_3', b_4']$.
\item For each of the following four cases {\bf (A)}, {\bf (B)}, {\bf (C)} and {\bf (D)}, proceed with the below steps (3-1)-(3-4):
\begin{enumerate}
\item[\bf (A)] Set $a_4$ $\leftarrow$ $1$, $\mathcal{S}$ $\leftarrow$ $\emptyset$, $x'$ $\leftarrow$ $X + a_1' W$, $y'$ $\leftarrow$ $Y + a_2' W$, $z'$ $\leftarrow$ $Z + a_3' W$, $w'$ $\leftarrow$ $W$ and $\boldsymbol{U}$ $\leftarrow$ $E_{11} - a_1' E_{14} + E_{22} - a_2' E_{24} + E_{33} - a_3' E_{34} + E_{44}$.
\item[\bf (B)] Set $a_4$ $\leftarrow$ $0$, $\mathcal{S}$ $\leftarrow$ $\{ a_3' - 1 \}$, $x'$ $\leftarrow$ $X + a_1' W$, $y'$ $\leftarrow$ $Y + a_2' W$, $z'$ $\leftarrow$ $W$, $w'$ $\leftarrow$ $Z$ and $\boldsymbol{U}$ $\leftarrow$ $E_{11} - a_1' E_{13} + E_{22} - a_2' E_{23} + E_{34} + E_{43}$.
\item[\bf (C)] Set $a_4$ $\leftarrow$ $0$, $\mathcal{S}$ $\leftarrow$ $\{ a_2' - 1, a_3' \}$, $x'$ $\leftarrow$ $X + a_1' W$, $y'$ $\leftarrow$ $W$, $z'$ $\leftarrow$ $Y$, $w'$ $\leftarrow$ $Z$ and $\boldsymbol{U}$ $\leftarrow$ $E_{11} - a_1' E_{12} + E_{23} +  E_{34} + E_{42}$.
\item[\bf (D)] Set $a_4$ $\leftarrow$ $0$, $\mathcal{S}$ $\leftarrow$ $\{ a_1' - 1, a_2', a_3' \}$, $x'$ $\leftarrow$ $W$, $y'$ $\leftarrow$ $X$, $z'$ $\leftarrow$ $Y$, $w'$ $\leftarrow$ $Z$ and $\boldsymbol{U}$ $\leftarrow$ $E_{12}+ E_{23} + E_{34} + E_{41}$.
\end{enumerate}
{\bf Note:} We denote by $E_{ij}$ the square matrix of size $4$ with the $(i,j)$ entry equal to $1$ and $0$ elsewhere.
Note also that $ (X, Y, Z, W) = (x', y', z', w') \cdot {}^t \boldsymbol{U}$.
\begin{enumerate}
\item[(3-1)] Substitute $(x',y',z',w')$ for $(x,y,z,w)$ in $R$, which we regard as an element in $S[X,Y,Z,W]$.
\item[(3-2)] Let $\mathcal{M}$ be the set of monomials in $S [X, Y, Z, W]$ of degree $3$, and set $\mathcal{S} \leftarrow \mathcal{S} \cup \{ (\mbox{the coefficient of $m$ in $R$} : \mbox{$m \in \mathcal{M}$ with $\deg_W (m) \geq 1$} \}$.
\item[(3-3)] With Gr\"{o}bner basis algorithms, compute the set $V (\mathcal{S}) \subset k^7$ of the zeros over $k$ of the ideal $\langle \mathcal{S} \rangle$ in $S$ generated by $\mathcal{S}$, where we construct the smallest extension field $K \subset k$ of $K_0$ such that $V (\mathcal{S}) \subset K^7 \subset k^7$.
If $[K : K_0] > 1$, represent all elements in $K_0$, which we have stored, as elements in $K$, and replace $\ell$ and $K_0$ respectively by $[K:K_0]$ and $K$.
\item[(3-4)] Set
\[
\mathrm{EQ}(H) \leftarrow \mathrm{EQ}(H) \cup \{ ( \mathbf{a}, \mathbf{b}, \boldsymbol{U}_{\mathbf{a},\mathbf{b}} ) : (a_1, a_2, a_3, b_1, b_2, b_3, b_4) \in V (\mathcal{S}) \},
\]
where $\mathbf{a}:= (a_1, a_2, a_3, a_4)$ and $\mathbf{b}:= (b_1, b_2, b_3, b_4)$, and where each $\boldsymbol{U}_{\mathbf{a},\mathbf{b}}$ is a matrix in $\mathrm{GL}_4 (k)$ obtained by substituting $(a_1, a_2, a_3, b_1, b_2, b_3, b_4)$ for $(a_1', a_2', a_3', b_1', b_2', b_3', b_4')$ in $\boldsymbol{U}$.
\end{enumerate}
{\bf Note:} It follows from Remark \ref{rem:EQ} that $\langle \mathcal{S} \rangle$ is of zero-dimension.
\item Output $\mathrm{EQ}(H)$, $\ell_H:=\ell_0$ and $K_H := K_0$.
\end{enumerate}
\end{algorithm}

\subsubsection{Pseudo code for Step (1-3) of Algorithm \ref{alg:main2}}\label{subsec:VH}
We use the same notation as in Subsections \ref{subsec:newIsom} and \ref{subsec:algorithm}.
Let $H:= V (Q, P) \subset \mathbf{P}^3$ be a Howe curve in canonical model.
For the set $\mathrm{EQ}(H)$ computed by Algorithm \ref{alg:sub1}, the following sub-algorithm \textsc{ComputeVH} (Algorithm \ref{alg:sub2}) computes a subset $V_{H}$ of $k^6$ with $k:=\overline{\mathbb{F}_{p}}$ required in Step (1-3) of Algorithm \ref{alg:main2} together with an integer $\ell_H^{\prime} \geq 1$ and a finite extension field $K_H^{\prime}$ of $K_H$ of degree $\ell_H^{\prime}$ such that $V_{H} \subset (K_H^{\prime})^6 \subset k^6$, where $V_H$ is the set of all $ v_H (U^{(1)},U^{(2)})$'s:
\begin{algorithm}\label{alg:sub2}
Sub-algorithm 2: \textsc{ComputeVH}($\mathrm{EQ}(H)$)
\begin{description}
\item[Input] A subset $\mathrm{EQ}(H) \subset (K_H)^8 \times \mathrm{GL}_4 (K_H)$ and an extension field $K_H$ of $\mathbb{F}_{p^2}$ computed by Algorithm \ref{alg:sub1}.
\item[Output] A subset $V_H \subset (\overline{\mathbb{F}_{p}})^6$, an integer $\ell_H^{\prime} \geq 1$, and a finite field $K_H^{\prime}$.
\end{description}
\begin{enumerate}
\item Set $V_{H}$ $\leftarrow$ $\emptyset$, $\ell$ $\leftarrow$ $1$ and $K$ $\leftarrow$ $K_H$.
\item For each pair of two distinct elements $(\mathbf{a}^{(i)}, \mathbf{b}^{(i)}, \boldsymbol{U}_{\mathbf{a}^{(i)}, \mathbf{b}^{(i)}} ) \in \mathrm{EQ}(H)$ with $i = 1$ and $2$, proceed with the following:
\begin{enumerate}
\item[(2-1)] Compute $R^{(i)}:= P-L^{(i)}\cdot Q$ with $L^{(i)} := b_1^{(i)}x+b_2^{(i)}y+b_3^{(i)}z+b_4^{(i)}w$.
\item[(2-2)] Let $\mathbf{u}_j^{(i)}$ denote the $i$-th row vector of $\boldsymbol{U}_{\mathbf{a}^{(i)}, \mathbf{b}^{(i)}}$ for $1 \leq j \leq 3$, and let $\mathbf{w}^{(i)}$ be the forth row vector of $\boldsymbol{U}_{\mathbf{a}^{(i)}, \mathbf{b}^{(i)}}$.

\noindent {\bf Note:} Namely $X^{(i)} = (x, y, z, w) \cdot {}^t \mathbf{u}_1^{(i)}$, $Y^{(i)} = (x, y, z, w) \cdot {}^t \mathbf{u}_2^{(i)}$, $Z^{(i)} = (x, y, z, w) \cdot {}^t \mathbf{u}_3^{(i)}$ and $W^{(i)} = (x, y, z, w) \cdot {}^t \mathbf{w}^{(i)}$.
\item[(2-3)] Compute a basis $\{ \mathbf{x}_1, \mathbf{x}_2 \} \subset K^4$ of the null space of dimension $2$ defined by $\mathbf{a}^{(i)} \cdot {}^t \mathbf{x} = 0$ for $i = 1$ and $2$ with $\mathbf{x} \in K^4$.
\item[(2-4)] Check whether $(U^{(1)}, U^{(2)} )$ is admissible (cf.\ Subsection \ref{subsec:newIsom}) or not by conducting the following:
For each of $i =1$ and $2$:
\begin{enumerate}
\item[(2-4-1)] For each $1 \leq j \leq 3$, test whether
\[
\mathrm{rank} \begin{pmatrix} {}^t \mathbf{x}_1 & {}^t \mathbf{x}_2 & {}^t \mathbf{u}_j^{(i)} \end{pmatrix}  = \mathrm{rank} \begin{pmatrix} {}^t \mathbf{x}_1 & {}^t \mathbf{x}_2 \end{pmatrix},
\]
i.e., $\mathbf{u}_j^{(i)} \in \langle \mathbf{x}_1, \mathbf{x}_2 \rangle_K$, or not.
If $\mathbf{u}_j^{(i)} \notin \langle \mathbf{x}_1, \mathbf{x}_2 \rangle_K$, set $\mathbf{z}^{(i)}$ $\leftarrow$ $\mathbf{u}_j^{(i)}$.
\item[(2-4-2)] Set $\boldsymbol{M}^{(i)} $ $\leftarrow$ $({}^t \mathbf{x}_1, {}^t \mathbf{x}_2, {}^t \mathbf{z}^{(i)}, {}^t \mathbf{w}^{(i)} ) $.
\item[(2-4-3)] Compute $\left( \boldsymbol{M}^{(i)} \right)^{-1}$, and substitute $(x_1, x_2, \underline{z}^{(i)}, W^{(i)}) \left( \boldsymbol{M}^{(i)} \right)^{-1} $ into $(x, y, z, w)$ in $R^{(i)}$, where $R^{(i)} \in K [x_1, x_2, \underline{z}^{(i)}, W^{(i)}]$ has no non-zero term containing $W^{(i)}$, and where we regard $R^{(i)}$ as a polynomial over $K[x_1,x_2]$ of with variable $\underline{z}^{(i)}$.
\item[(2-4-4)] If $\deg_{\underline{z}^{(i)}} (R^{(i)}) =3$ (i.e., $(U^{(1)}, U^{(2)})$ is not admissible), then go back to Step (2) and choose another pair of $(\mathbf{a}^{(i)}, \mathbf{b}^{(i)}, \boldsymbol{U}_{\mathbf{a}^{(i)}, \mathbf{b}^{(i)}} ) \in \mathrm{EQ}(H)$ with $i = 1$ and $2$.
\end{enumerate}
\item[(2-5)] Compute $\mathbf{y} := \mathbf{b}_1 - \mathbf{b}_2$.

\noindent {\bf Note:} We have $\underline{y} = (x, y, z, w) \cdot {}^t \mathbf{y}$, where $\underline{y}$ is as in Lemma \ref{lem:y}.
\item[(2-6)] Find an element $\underline{x} = \xi_1 x + \xi_2 y + \xi_3 z + \xi_4 w \in U^{\rm (1)} \cap U^{\rm (2)}$ such that $\underline{x}$ and $\underline{y}$ are linearly independent over $K$, by the following procedure:
\begin{itemize}
\item For $j =1$ and $2$, test the linear independence of $\mathbf{x}_j$ and $\mathbf{y}$.
If $\mathbf{x}_j$ and $\mathbf{y}$ are linearly independent, set $\boldsymbol{\xi} $ $\leftarrow$ $\mathbf{x}_j$ and write $\boldsymbol{\xi} = (\xi_1, \xi_2, \xi_3, \xi_4 ) $.
\end{itemize}
{\bf Note:} For one or the other of $\mathbf{x}_1$ and $\mathbf{x}_2$, the vectors $\mathbf{x}_j$ and $\mathbf{y}$ are linearly independent.
Note also that $\underline{x} = \xi_1 x + \xi_2 y + \xi_3 z + \xi_4 w $, where $\underline{x}$ is as in the proof of Lemma \ref{lem:y}.
\item[(2-7)] For each of $i =1$ and $2$, conduct the following:
\begin{enumerate}
\item[(2-7-1)] For each $1 \leq j \leq 3$, test whether
\[
\mathrm{rank} \begin{pmatrix} {}^t \mathbf{x}_1 & {}^t \mathbf{x}_2 & {}^t \mathbf{u}_j^{(i)} \end{pmatrix}  = \mathrm{rank} \begin{pmatrix} {}^t \mathbf{x}_1 & {}^t \mathbf{x}_2 \end{pmatrix},
\]
i.e., $\mathbf{u}_j^{(i)} \in \langle \mathbf{x}_1, \mathbf{x}_2 \rangle_K$, or not.
If $\mathbf{u}_j^{(i)} \notin \langle \mathbf{x}_1, \mathbf{x}_2 \rangle_K$, set $\mathbf{z}^{(i)}$ $\leftarrow$ $\mathbf{u}_j^{(i)}$.
\item[(2-7-2)] Set $\boldsymbol{T}^{(i)} $ $\leftarrow$ $({}^t \boldsymbol{\xi}, {}^t \mathbf{y}, {}^t \mathbf{z}^{(i)}, {}^t \mathbf{w}^{(i)} ) $.
\item[(2-7-3)] Compute $\left( \boldsymbol{T}^{(i)} \right)^{-1}$.
Substitute $(\underline{x}, \underline{y}, \underline{z}^{(i)}, W^{(i)}) \left( \boldsymbol{T}^{(i)} \right)^{-1} $ with $\underline{y}:=1$ into $(x, y, z, w)$ in $R^{(i)}$, where $R^{(i)}$ has no non-zero term containing $W^{(i)}$.
\item[(2-7-4)] Regarding $R^{(i)}$ as a polynomial over $K[\underline{x}]$ of degree two with variable $\underline{z}^{(i)}$, compute its discriminant $D^{(i)} \in K[\underline{x}]$.
\item[(2-7-5)] Test whether $D^{(i)}$ splits completely over $K$ or not.
If not, construct the splitting field $K^{\prime}$ of $D^{(i)}$.
Represent all elements in $K$, which we have stored, as elements in $K^{\prime}$.
Replace $K$ by $K^{\prime}$.
\item[(2-7-6)] Compute the roots $\underline x_1^{(i)}$, $\underline x_2^{(i)}$ and $\underline x_3^{(i)}$ over $K$ of $D^{(i)}$.
\end{enumerate}
{\bf Note:} We have $\underline{z}^{(i)} \notin U^{(1)} \cap U^{(2)}$ if one puts $\underline{z}^{(i)}:= (x,y,z,w) \cdot {}^t \mathbf{z}^{(i)} \in \{ X^{(i)}, Y^{(i)}, Z^{(i)} \}$.
Note also that $(\underline{x}, \underline{y}, \underline{z}^{(i)}, W^{(i)}) = (x, y, z, w) \cdot \boldsymbol{T}^{(i)} $.
\item[(2-8)] Set $v_H(U^{(1)},U^{(2)})$ $\leftarrow$ $(\underline x_1^{(1)},\underline x_2^{(1)},\underline x_3^{(1)},\underline x_1^{(2)},\underline x_2^{(2)},\underline x_3^{(2)})$, and we set $V_H$ $\leftarrow V_H$ $\cup \{ v_H (U^{(1)},U^{(2)}) \}$ .
\end{enumerate}
\item Output $V_H$.
\end{enumerate}
\end{algorithm}

\section{Concluding remark}\label{sec:conc}

In this paper, we presented the two main algorithms to find/enumerate superspecial Howe curves (and thus canonical curves of genus $4$).
The first one determines the (non-)existence of such curves by collecting their explicit defining equations.
The second one enumerates the isomorphism classes of superspecial Howe curves listed by the first one.
By executing these algorithms over Magma, we have a result that there exists a superspecial Howe curve for every prime $p=5$ or $7 < p \leq 331$, and we determined their $\overline{\mathbb{F}_p}$-isomorphism classes for all $5 \leq p \leq 53$.

For the construction of the main algorithms, several sub-algorithms were also constructed.
In particular one of them is an algorithm to test whether given two Howe curves are isomorphic (over an algebraically closed field) to each other or not.
While the complexity of this isomorphism-test has not been analyzed yet, experimental results show that it is expected to classify the isomorphism classes extremely more efficiently than the conventional method by Kudo-Harashita, 2017~\cite{KH17}.
Benchmarking shows that our algorithm classified isomorphisms classes $15$-$214$ times (on average $96$ times) faster than the conventional method.

We conclude that the main algorithms shall contribute to solve the problem to find or enumerate superspecial curves for larger $p$.
We also hope that algorithms proposed in this paper could be generalized and applied to other curves obtained by fiber products.

